\documentclass{article}
\pdfoutput=1

\usepackage[colorlinks,linkcolor=blue,citecolor=blue]{hyperref}
\usepackage{amsmath,amssymb,amscd,amsfonts,amsthm}
\usepackage[numbers]{natbib}
\usepackage[letterpaper,top=1.5in,bottom=1in,left=1.6in,right=1.6in]{geometry}

\theoremstyle{plain}
\newtheorem{theorem}{Theorem}
\newtheorem{lemma}[theorem]{Lemma}

\date{\empty}

\begin{document}

\title{\Large\bf Borel Liftings of Graph Limits}
\author{
  \begin{minipage}{.65\textwidth}
    \begin{minipage}[c][1cm][t]{0.5\textwidth}
      \begin{center}
        \normalsize Peter Orbanz\\[2.5pt]
        \emph{Columbia University}
      \end{center}
    \end{minipage}\hfill
    \begin{minipage}[c][1cm][t]{0.5\textwidth}
      \begin{center}
        \normalsize Bal\'{a}zs Szegedy \\[2.5pt]
        \emph{Alfr\'{e}d R\'{e}nyi Institute\\ of Mathematics}
      \end{center}
    \end{minipage}\hfill
  \end{minipage}
}

\maketitle

\def\environ{U}
\def\corresp{\phi_{_{\square}}}
\def\weakinv{\corresp^{-1}}
\def\graphonspace{\mathbf{W}}
\def\quotientspace{\widehat{\graphonspace}}
\def\dOne{d_1}
\def\dCut{d_{\square}}
\def\deltaOne{\delta_1}
\def\deltacut{\delta_{_{\square}}}
\def\rearrangement{\psi}
\def\selector{\xi}
\def\eqvclass#1{[#1]_{_{\square}}}
\def\borel{\mathcal{B}}
\def\relation{\equiv_{_{\square}}}

\begin{abstract}
  The cut pseudo-metric on the space of graph limits
  induces an equivalence relation. The quotient space
  obtained by collapsing each equivalence class to a
  point is a metric space with appealing analytic properties.
  We show that the equivalence relation admits a Borel lifting: There
  exists a Borel-measurable mapping which maps each equivalence class to one
  of its elements.
\end{abstract}

\section*{\empty}

Let $(\Omega,\borel(\Omega),P)$ be an atomless Borel probability space
and $L_1(\Omega^2)$ the Banach space of integrable
functions on $\Omega\times\Omega$, equipped with the $L_1$-metric $\dOne$.
Let $\graphonspace\subset L_1(\Omega^2)$ be the subspace of symmetric integrable functions $\Omega^2\rightarrow[0,1]$.
Define a pseudo-norm on $\graphonspace$ by
\begin{equation}
  \|w\|_{_{\square}}
    =
    \sup_{S,T\in\borel(\Omega)}\int_{S\times T}w(s,t)dP(s)dP(t) \;.
\end{equation}
Following \citep{Borgs:Chayes:Lovasz:Sos:Vesztergombi:2008},
we use $\|\,.\,\|_{_{\square}}$ to define a pseudo-metric on $\graphonspace$ as
\begin{equation}
  \deltacut(w,w')
  :=
  \inf_{\psi}\|w^{\psi}-w'\|_{_{\square}} \qquad\text{where}\qquad
  w^{\psi}(x,y)=w(\psi(x),\psi(y))\;.
\end{equation}
The infimum is taken over all invertible measure-preserving transformations of $\Omega$, i.e.\ all
invertible measurable mappings ${\psi:\Omega\rightarrow\Omega}$ satisfying $\psi P=P$.
The pseudo-metric induces an equivalence relation on $\graphonspace$, given by 
${w\equiv w'\Leftrightarrow\deltacut(w,w')=0}$. The relation
${w\equiv w'}$ is also known as \emph{weak isomorphy} of $w$ and $w'$ \citep{Lovasz:2013}.
Denote the equivalence class of ${w\in\graphonspace}$
by $\eqvclass{w}$, and the quotient space of all equivalence classes by $\quotientspace$. 
On the quotient space, $\deltacut$ is a metric, and the metric space $(\graphonspace,\deltacut)$ is known
to be compact \citep[][]{Lovasz:Szegedy:2007}.
For each ${\widehat{w}\in\quotientspace}$, we write
${\eqvclass{\widehat{w}}\subset\graphonspace}$ for the corresponding
equivalence class of elements of $\graphonspace$.

Theorem \ref{theorem:result} below shows that weak isomorphy admits a Borel lifting,
i.e.\ there exists a 
Borel-measurable mapping 
${\xi:(\quotientspace,\deltacut)\rightarrow(\graphonspace,\dOne)}$ such that
\begin{align}
  \label{eq:result}
  {} &&
  \xi(\widehat{w})\in\eqvclass{\widehat{w}} && \text{ for all } \widehat{w}\in\quotientspace\;.
\end{align}
The lifting is not unique. More precisely:
\begin{theorem}
  \label{theorem:result}
  There is a sequence ${(\xi_n)}$ of measurable mappings
  ${\xi_n\!:(\quotientspace,\deltacut)\rightarrow(\graphonspace,\dOne)}$
such that, for every $\widehat{w}\in\quotientspace$, the set
  ${\lbrace \xi_n(\widehat{w})\,|\,n\in\mathbb{N}\rbrace}$ is a dense
  subset of $\eqvclass{\widehat{w}}$.
\end{theorem}

\newpage

{\noindent\bf Proof.} 
Theorem \ref{theorem:result} can be stated equivalently by
defining a set-valued mapping
\begin{equation}
  \label{eq:def:correspondence}
  \corresp:\quotientspace\rightarrow 2^{\graphonspace}
  \qquad
  \text{ with }
  \qquad
  \corresp(\widehat{w}):=\eqvclass{\widehat{w}}\;.
\end{equation}
We then have to show that there are measurable mappings $\xi_n$ with 
\begin{equation}
  \overline{\lbrace \xi_n(\widehat{w})\,\vert\,n\in\mathbb{N}\rbrace}
  =\corresp(\widehat{w}) \qquad\text{ for all }\widehat{w}\in\quotientspace\;,
\end{equation}
where $\overline{A}$ denotes the closure of $A$. 

Liftings of set-valued maps are a well-studied topic in analysis, and
we use a result of \citet{Kuratowski:Ryll-Nardzewski:1965} on the
existence of liftings, and a generalization by \citet{Castaing:1967}
(see e.g.\ \citep[][Theorem 12.16]{Kechris:1995} and \citep[][Theorem 14.4.1]{Klein:Thompson:1984}
for textbook statements).
For our purposes, these results can be summarized as follows:
\begin{theorem}
  \label{theorem:KRN}
  Let $\mathbf{X}$ be a measurable space, $\mathbf{Y}$ a Polish space,
  and 
  ${\phi:\mathbf{X}\rightarrow 2^{\mathbf{Y}}}$ a set-valued mapping.
  Require ${\phi(x)}$ to be non-empty and closed for all ${x\in\mathbf{X}}$, and that
  \begin{equation}
    \label{eq:def:weakinv}
    \phi^{-1}(A):=\lbrace x\in \mathbf{X}\,|\,\phi(x)\cap A\neq\emptyset\rbrace
  \end{equation}
  is a measurable set in $\mathbf{X}$ for each open set $A$ in $\mathbf{Y}$. 
  Then there exists a sequence of measurable mappings ${\selector_n:\mathbf{X}\rightarrow \mathbf{Y}}$
  such that ${\overline{\lbrace \selector_n(x)\,|\,n\in\mathbb{N}\rbrace}=\phi(x)}$
 for all ${x\in \mathbf{X}}$.
\end{theorem}

\noindent
For $\corresp$ as defined in \eqref{eq:def:correspondence}
and any subset
${A\subset\graphonspace}$, 
the set ${\weakinv(A)}$ 
in \eqref{eq:def:weakinv} simply consists of all ${\widehat{w}\in\quotientspace}$ 
for which $A$ contains at least one element of the equivalence class
$\eqvclass{\widehat{w}}$. If $A$ is in particular an
open $\dOne$-ball in $\graphonspace$, this set has the following
property:

\begin{lemma}
  \label{lemma:closure}
  Denote by $U_{\varepsilon}(v)$ the open $\dOne$-ball of radius $\varepsilon$ centered at $v\in\graphonspace$.
  If $\varepsilon<\delta$, 
  \begin{equation}
    \label{eq:lemma:closure}
    \overline{\weakinv(U_{\varepsilon}(v))}\;\subseteq\;\weakinv(U_{\delta}(v))
  \end{equation}
  for all $v\in\graphonspace$.
\end{lemma}

\begin{proof}[Proof of Lemma \ref{lemma:closure}]
  Let ${(\widehat{w}_1,\widehat{w}_2,\ldots)}$ be a sequence in
  $\weakinv(U_{\varepsilon}(v))$ with 
  ${\widehat{w}_i\xrightarrow{\deltacut}\widehat{w}}$.
  We have to show that $\widehat{w}\in \weakinv(U_{\delta}(v))$.
  By definition of $\weakinv$, the sets 
  ${\corresp(\widehat{w}_i)\cap U_{\varepsilon}(v)}$ are non-empty. 
  Suppose $(w_i)$ is a sequence with
  ${w_i\in\corresp(\widehat{w}_i)\cap U_{\varepsilon}(v)}$ for each
  ${i\in\mathbb{N}}$. By 
  \citep[][Lemma 2.11]{Lovasz:Szegedy:2010}, convergence of
  $(\widehat{w}_i)$ to $\widehat{w}$ then implies
  \begin{equation}
    \label{eq:proof:closure:1}
    \varepsilon \geq \liminf\deltaOne(w_i,v) \geq \deltaOne(w,v) = \inf_{\rearrangement}\dOne(w^{\rearrangement},v)
  \end{equation}
  for any ${w\in\corresp(\widehat{w})}$.
  Since ${\varepsilon<\delta}$, there is hence a 
  measure-preserving transformation $\rearrangement$ such that
  ${\dOne(w^{\psi},v)<\delta}$, that is, ${w^{\psi}\in U_{\delta}(v)}$.
  Because $w^{\rearrangement}$ and $w$
  are weakly isomorphic, we also have 
  ${w^{\rearrangement}\in\corresp(\widehat{w})}$, and
  therefore
  \begin{equation}
    \widehat{w}
    \;\in\;
    \weakinv(\corresp(\widehat{w})\cap U_{\delta}(v))
    \;\subset\;
    \weakinv(U_{\delta}(v)) \;.
  \end{equation}
\end{proof}

\begin{proof}[Proof of Theorem \ref{theorem:result}]
  The space $(W,\dOne)$ is a closed subspace of the separable Banach space $L_1(\Omega^2)$ and hence Polish.
  The sets $\corresp(\widehat{w})$ are non-empty, by definition of the space $\quotientspace$ as a quotient.
  We will show that, additionally:
  \begin{enumerate}
    \renewcommand{\labelenumi}{\roman{enumi}.}
  \item The sets $\corresp(\widehat{w})$ are closed.
  \item For all open sets $A$ in $\graphonspace$, the set $\weakinv(A)$ is Borel in $\quotientspace$.
  \end{enumerate}
  The mapping $\corresp$ therefore satisfies the hypothesis of Theorem
  \ref{theorem:KRN}, and Theorem \ref{theorem:result} follows.\\
  
  {\noindent (i)}
  Denote by ${t_{F}\!:\graphonspace\rightarrow[0,1]}$ the homomorphism
  density indexed by a finite graph $F$ \citep[][]{Lovasz:Szegedy:2006}. 
  Two elements of $\graphonspace$ are
  weakly isomorphic if and only if their homomorphism densities coincide for all finite graphs $F$.
  Let ${\widehat{w}\in\quotientspace}$, and let ${(w_1,w_2,\ldots)}$
  be a sequence in the set ${\phi(\widehat{w})}$ with limit $w$ in $(\graphonspace,\dOne)$.
  The homomorphism densities are $\deltacut$-continuous and hence $\dOne$-continuous. Therefore,
  \begin{equation}
    \lim t_{F}(w_i)=t_F(w)\qquad\text{ for all } F\;,
  \end{equation}
  and since the $w_i$ are weakly isomorphic, ${t_F(w_i)=t_F(w)}$ for all $i$
  and all $F$. Thus, ${w\in\phi(\widehat{w})}$, and the set is closed.\\

  {\noindent (ii)}
  Let $U_{\delta}(v)$ denote the open ball of radius $\delta$ centered at $v\in\graphonspace$.
  Since $W$ is Polish, the open balls form a base of the topology and it is sufficient to consider sets of the form
  $A=U_{\delta}(v)$.
  Let $\delta_i\in\mathbb{R}_+$ be an increasing sequence $\delta_i\rightarrow\delta$. Then, by Lemma \ref{lemma:closure},
  \begin{equation*}
    \weakinv(U_{\delta}(v))
    \;=\;
    \bigcup_{i}\weakinv(U_{\delta_i}(v))
    \;\subseteq\;
    \bigcup_{i}\overline{\weakinv(U_{\delta_i}(v))}
    \;\stackrel{\eqref{eq:lemma:closure}}{\subseteq}\;
    \bigcup_{i}\weakinv(U_{\delta}(v))
    \;=\;
    \weakinv(U_{\delta}(v))\;.
  \end{equation*}
  In particular, $\weakinv(U_{\delta}(v))$ is a countable union of the closed sets
  $\overline{\weakinv(U_{\delta_i}(v))}$, and hence Borel.
\end{proof}

\end{document}